\documentclass{amsart}

	\usepackage{amsmath,amssymb,amsthm,wrapfig,amsfonts,enumerate,latexsym}

	\newtheorem{dfn}{Definition}[section]
	\newtheorem{thm}[dfn]{Theorem}
	\newtheorem{prop}[dfn]{Proposition}
	\newtheorem{cor}[dfn]{Corollary}
	\newtheorem{lem}[dfn]{Lemma}
	\newtheorem{rem}[dfn]{Remark}

	\newtheorem{ack}{Acknowledgements\!\!}

\theoremstyle{definition}

\theoremstyle{remark}

\numberwithin{equation}{section}

	\newcommand{\dist}{\mathop{\mathit{d}} \nolimits}
	
	\newcommand{\diam}{\mathop{\mathrm{diam}} \nolimits}

        \newcommand{\e}{\mathop{\varepsilon}              \nolimits}

\begin{document}

\title[Embedding proper ultrametric spaces into $\ell_p$]{Embedding
proper ultrametric spaces into $\ell_p$ and
    its application to nonlinear Dvoretzky's theorem}

\author[Kei Funano]{Kei Funano}
\address{Research Institute for Mathematical Sciences, Kyoto University,
Kyoto 606-8502 JAPAN}
\email{kfunano@kurims.kyoto-u.ac.jp}
\thanks{This work was partially supported by Grant-in-Aid for
    Research Activity (startup) No.23840020.}


\subjclass[2000]{53C23}

\date{\today}


\keywords{Dvoretzky's theorem, embedding, ultrametric space}

\begin{abstract}We prove that every proper ultrametric space
 isometrically embeds into
 $\ell_p$ for any $p\geq 1$. As an application we discuss an $\ell_p$-version of nonlinear Dvoretzky's theorem.
\end{abstract}

\maketitle


\section{Introduction}
Recall that a metric space $(X,\rho)$ is called an
 \emph{ultrametric space} if for every $x,y,z\in X$ we have $\rho(x,y)
 \leq \max\{ \rho (x,z), \rho(z,y) \}$. Such spaces naturally appear and
 have applications in
 various areas such as number theory, $p$-adic analysis, and computer
 science (see \cite{Lem1}, \cite{Lemin}, \cite{Mo1, Mo2}).

  Let us briefly review several results with respect to isometric embedding
 of ultrametric spaces. Timan
 and Vestfrid \cite{VeTi, VeTi2} proved that any separable ultrametric space embed
 isometrically into $\ell_2$. Vestfrid \cite{Ves2} later proved that the result is also true if one
 replace $\ell_2$ by $\ell_1$ and $c_0$ by constructing a universal
 ultrametric space for the class of separable ultrametric space and
 using its property. Vestfrid \cite{Ves} also proved that a certain
class of countable ultrametric spaces embed isometrically into $\ell_p$
 for $p\geq 1$.
Lemin \cite{Lemin} proved that any separable ultrametric space embed isometrically
 into the Lebesgue space. He also raised a problem whether any separable
 ultrametric space embed isometrically into any infinite dimensional
 Banach space. Motivated by Lemin's problem, Shkarin \cite{Shk} proved that every finite ultrametric space embeds into every
infinite dimensional Banach space. From these results ultrametric spaces
 have attracted much attention in embedding theory.
 
 In this paper we tackle Lemin's problem in the case where the target
 Banach space is $\ell_p$. It is already well-known that every separable ultrametric space embeds
 isometrically into the function space $L_p$ for any $p\geq 1$. In fact,
 it follows from Timan and Vestfrid's result mentioned above and the
 fact that $\ell_2$ embeds isometrically into $L_p$. Since $\ell_2$ does
 not embed bi-Lipschitzly into $\ell_p$ for any $p\neq 2$
 (\cite[Corollary 2.1.6]{alkal}), embedding separable ultrametric spaces
 into $\ell_p$ is left as a problem. Our main theorem is the following: Recall that a metric space is \emph{proper} if every
 closed ball in $X$ is compact. 
 \begin{thm}\label{MT}Every proper ultrametric space isometrically embeds into
  $\ell_p$ for any $p\geq 1$. 
  \end{thm}

  The case of general separable ultrametric spaces remains open. A similar method of the proof of Theorem \ref{MT} also implies an
   isometric embedding into $c_0$ (see Remark \ref{rem1}). 
   Our construction of isometric embeddings into $\ell_1$, $\ell_2$, and
   $c_0$ is different from the one by \cite{VeTi, VeTi2}, \cite{Ves, Ves2} in the
   case of proper ultrametric spaces. 

   As an application of Theorem \ref{MT} we obtain an $\ell_p$-version of
   nonlinear Dvoretzky's theorem, see Section $3$. 

 \section{Proof of the main theorem}

 We use some basic
  facts of compact ultrametric spaces (see \cite{hugh}, \cite[Section 2]{MN112}). Let $(X,\rho)$ be a compact ultrametric space and put $r_0:=\diam X$.
  Consider the relation $\sim_0$ on $X$ given by $x \sim_0 y$
  $\Longleftrightarrow $ $\rho(x,y)< r_0$. Since $\rho$ is ultrametric
  $\sim_0$ is an equivalence relation on $X$. The compactness of $X$
  implies that each equivalence
  class is a closed ball of radius strictly less than $r_0$ (see
  \cite[Section $2$]{MN112}). Since the distance between two distinct equivalence classes
  is exactly $r_0$ and $X$ is totally bounded, there are only finitely
  many equivalence
  classes, say, $\{B_1, \cdots, B_{k_1}\}$, where each $B_i$ is a closed
  ball of radius $r_i=\diam B_i<r_0$. Note that for any $x\in B_i$ and $y\in
  B_j$ $(i\neq j)$ we have $\rho(x,y)=r_0$. For each $i$ we choose
  $x_i\in B_i$ and fix it. As above for each $i_1=1,\cdots,k_1$ we
  consider the equivalence relation $\sim_{i_1}$ on $B_{i_1}$ given by $x \sim_{i_1} y$
  $\Longleftrightarrow $ $\rho(x,y)< r_{i_1}$. Then we
  can divide $X_{i_1}$ into finitely many equivalence classes, i.e., $B_{i_1}=
  \amalg_{i_2=1}^{k(i_1)} B_{i_1 i_2}$, where $B_{i_1 i_2}$ is a closed ball of
  radius $r_{i_1 i_2}=\diam B_{i_1 i_2}<r_{i_1}$. We may assume that
  $x_{i_1}\in B_{i_1 1}$. For each $i_1,i_2$, we choose a point
  $x_{i_1i_2}\in B_{i_1 i_2}$ so that $x_{i_1 1}= x_{i_1}$ and we fix
  $x_{i_1 i_2}$. Repeatedly we get a sequence $\mathcal{P}_k= \{ B_{i_1 \cdots i_k}
  \}_{i_1,\cdots, i_k}$ of partitions of $X$ satisfying
  the following: 

  \begin{enumerate}
   \item Each $B_{i_1 \cdots i_k}$ is a closed ball of radius $r_{i_1
         \cdots i_k}=\diam B_{i_1 \cdots i_k}$.
   \item If $r_{i_1 \cdots i_k}\neq 0$, then $r_{i_1 \cdots i_k}> r_{i_1
         \cdots i_{k+1}}$.
   \item $B_{i_1 \cdots i_{k-1}}= \amalg_{i_k}B_{i_1 \cdots i_{k-1}i_k}$.
   \end{enumerate}For each $i_1, \cdots, i_k$ we choose $x_{i_1\cdots i_k}\in B_{i_1\cdots i_k}$ so that
  $x_{i_1 \cdots i_k 1 \cdots 1}=x_{i_1 \cdots i_k}$. The
  compactness of $X$ yields the following:

  \begin{lem}[{cf.~\cite[Section 2]{MN112}}]\label{fact1}$\lim_{k\to \infty} \max_{i_1,\cdots, i_k} r_{i_1 \cdots i_k}=0 $.
   \end{lem}In particular, $\bigcup_{k=1}^{\infty}\{x_{i_1 \cdots i_k}
  \}_{i_1,\cdots,i_k} $ is a countable dense subset of $X$.

  \begin{lem}[{cf.~\cite[Section 2]{MN112}}]\label{fact2} For every closed
   ball $B$ in $X$, there exist $k$ and $B_{i_1 \cdots i_k}\in
   \mathcal{P}_k$ such that $B=B_{i_1\cdots i_k}$.
   \end{lem}
   \begin{proof}[Proof of Theorem \ref{MT}]
    We first prove the theorem for compact ultrametric
    spaces. Let $(X,\rho)$ be a compact ultrametric space and let
    $\mathcal{P}_k = \{ B_{i_1 \cdots i_k} \}_{i_1,\cdots, i_k}$, $r_{i_1 \cdots i_k}$, and
    $x_{i_1 \cdots i_k}$ as above. Put $N_k:= \# \mathcal{P}_k$. We consider each coordinate of an element of
  $\ell_p^{N_k}$ is indexed by $(i_1, \cdots, i_k)$. We define a map
  $f_k: \{ x_{i_1 \cdots i_k} \}_{i_1 , \cdots , i_k}\to \ell_p^{N_k}$
  as follows: $(f_k(x_{i_1 \cdots  i_k}))_{(j_1, \cdots, j_{k})}: = 0$
  if $(j_1,\cdots, j_k) \neq (i_1,\cdots, i_k)$ and 
  \begin{align*}(f_1(x_{i_1}))_{i_1}:=\frac{(r_0^p-r_{i_1}^p)^{\frac{1}{p}}}{2^{\frac{1}{p}}}
   \text{ and }
   (f_k(x_{i_1 \cdots i_k}))_{(i_1, \cdots, i_k)}:=
  \frac{(r_{i_1 \cdots i_{k-1}}^p-r_{i_1 \cdots
  i_k}^p)^{\frac{1}{p}}}{2^{\frac{1}{p}}} \text{ if }k\geq 2.\end{align*}
  Note that $f_k(x_{i_1\cdots i_k})\perp f_k(x_{j_1 \cdots j_k})$ for two distinct $(i_1,\cdots , i_k)$, $(j_1,\cdots, j_k)$.

  We define a map $f:\bigcup_{k=1}^{\infty} \{
  x_{i_1  \cdots i_k}\}_{i_1, \cdots, i_k} \to \ell_p $ as
  follows. For each $x_{i_1  \cdots i_k}$, putting
  $i_m:=1$ for $m>k$, we define
  \begin{align*}
    f(x_{i_1 i_2 \cdots i_k}) := 
   (f_1(x_{i_1}), f_2(x_{i_1 i_2}), \cdots , f_m(x_{i_1\cdots i_m}),
   \cdots ).
   \end{align*}The right-hand side in the above definition is
  actually the element of $\ell_p$ since
  \begin{align*}
   \sum_{m=1}^{\infty} \| f_m(x_{i_1 \cdots i_m})\|_p^p=
   \sum_{m=1}^{\infty} \frac{r_{i_1 \cdots i_{m-1}}^p - r_{i_1 \cdots
   i_m}^p}{2} = \frac{r_{0}^p}{2} <+\infty
   \end{align*}by Lemma \ref{fact1}. Note that $f$ is well-defined in the
  sense that $f(x_{i_1 \cdots
  i_k 1 \cdots  1})=f(x_{i_1 \cdots i_k})$.

  We shall prove that $f$ is an isometric embedding. Since $\bigcup_{k=1}^{\infty} \{
  x_{i_1  \cdots i_k}\}_{i_1, \cdots, i_k}$ is dense in $X$ this implies
  the theorem. Taking two distinct elements $x_{i_1 \cdots i_k}$ and $x_{j_1 \cdots
  j_l}$ we may assume that $k\leq l$. Put $i_m:=1 $ for $m>k$. Then we have $(i_1,\cdots, i_l)\neq (j_1,\cdots,j_l)$. Letting
  \begin{align*}
   n:=\min \{ m \leq l \mid i_m \neq j_m\}
   \end{align*}we get
  $\rho(x_{i_1 \cdots i_k}, x_{j_1 \cdots j_l})= \diam B_{i_1 \cdots
  i_{n-1}}=r_{i_1 \cdots i_{n-1}}$ if $n\geq 2$ and $\rho(x_{i_1 \cdots
  i_k}, x_{j_1 \cdots j_l})=r_0$ if
  $n=1$. Since $f_m(x_{i_1 \cdots i_m})=f_m(x_{j_1 \cdots j_m})$ for $m<n$ and
  $f_m(x_{i_1 \cdots i_m}) \perp f_m(x_{j_1 \cdots j_m})$ for $m\geq n$,
  \begin{align*}
   \| f(x_{i_1 \cdots i_k})-f(x_{j_1 \cdots j_l})\|_p^p=\ &
   \sum_{m=n}^{\infty}\|f(x_{i_1 \cdots i_m}) \|_p^p +
   \sum_{m=n}^{\infty} \| f(x_{j_1 \cdots j_m}) \|_p^p\\
   =\ & r_{i_1
   \cdots i_{n-1}}^p\\  =\ & \rho (x_{i_1 \cdots i_k},x_{j_1 \cdots j_l})^p.
   \end{align*}
  This completes the proof of the theorem for compact ultrametric spaces.

    Let $(X,\rho)$ be a proper ultrametric space and fix a point $x_0\in
    X$. For any $r>0$ we denote by $B(x_0,r)$ the closed ball of radius
    $r$ centered at $x_0$. For any $R>0$ let
    $f_1:B(x_0,R)\to \ell_p$ be an isometric embedding constructed as in
    the above way. It suffices to prove that for any $R'>R$ we can
    construct an isometric embedding $f_2:B(x_0,R')\to \ell_p$ as in the above way, which
    extends $f_1$ in the following sense: There exists an isometry
    $T:\ell_p \to \ell_p$ such that $T \circ f_2|_{B(x_0,R)}=f_1$.
    This is possible by the above construction. In fact, keep dividing
    $B(x_0,R')$ as in the above way. Then at finite steps we reach at $B(x_0,R)$ by
    Lemma \ref{fact2} since $B(x_0,R')$ is compact. From the above
    construction we easily see the existence of $f_2$ and $T$ . This
    completes the proof of the theorem.
 \end{proof}

 \begin{rem}\label{rem1}\upshape  A similar method of the above proof implies new
  isometric embeddings of proper ultrametric spaces into $c_0$. In fact,
  let us consider first the case of compact ultrametric spaces. Using the same notation as above, for each $k$ we define $g_k: \{ x_{i_1 \cdots i_k}\}_{i_1, \cdots, i_k}\to
    \ell_{\infty}^{N_k}$ as follows: $(g_k(x_{i_1 \cdots  i_k}))_{(j_1, \cdots, j_{k})}: = 0$
  if $(j_1,\cdots, j_k) \neq (i_1,\cdots, i_k)$ and 
  \begin{align*}(g_1(x_{i_1}))_{i_1}:=r_0
   \text{ and }
   (g_k(x_{i_1 \cdots i_k}))_{(i_1, \cdots, i_k)}:=
  r_{i_1 \cdots i_{k-1}} \text{ if }k\geq 2.\end{align*}Then we define a
  map $g:\bigcup_{k=1}^{\infty} \{ x_{i_1 \cdots i_k}\}_{i_1,\cdots,
  i_k}\to c_0$ by
  \begin{align*}
   g(x_{i_1 i_2 \cdots i_k}):= (g_{1}(x_{i_1}), g_2(x_{i_1 i_2} ),
   \cdots, g_m(x_{i_1  \cdots i_m}),\cdots),
   \end{align*}where as in the above proof we put $i_m:=1$ for
  $m>k$. Note that the right-hand side of the above definition is in
  $c_0$ by Lemma \ref{fact1}. We can easily check that the map
  $g:\bigcup_{k=1}^{\infty} \{ x_{i_1 \cdots i_k} \}_{i_1, \cdots, i_k}
  \to c_0$ is an isometric embedding. As in the proof of Theorem \ref{MT} this construction also implies an
  isometric embedding from every proper ultrametric space into $c_0$.
  \end{rem}

   \section{$\ell_p$-version of nonlinear Dvoretzky's theorem}

   In this section we apply Theorem \ref{MT} to obtain an
   $\ell_p$-version of nonlinear Dvoretzky's theorem. Refer to
   \cite{blmn2}, \cite{ChKa} for the case of finite metric spaces.

   We say that a metric space $X$ is \emph{embedded with distortion} $D\geq 1$
in a metric space $Y$ if there exist a map $f:X\to Y$ and a constant $r>0$ such that
\begin{align*}
 r \dist_X(x,y) \leq \dist_Y(f(x),f(y))\leq D r \dist_X(x,y) \text{ for all
 }x,y\in X.
 \end{align*}

 Dvoretzky's theorem states that for every $\e>0$, every
 $n$-dimensional normed space contains a $k(n,\e)$-dimensional subspace
 that embeds into a Hilbert space with distortion $1+\e$
 (\cite{Dvo60}). This theorem was conjectured by Grothendieck
 (\cite{groth}). See \cite{Mil71} and \cite{MilSch96}, \cite{Sch} for
 the estimate of $k(n,\e)$. Bourgain, Figiel, and Milman \cite{BFM86} first studied Dvoretzky's
 theorem in the nonlinear setting. They obtained that for every $\e>0$, every finite metric
 space $X$ contains a subset $S$ of sufficiently large size which embeds into a Hilbert space with
 distortion $1+\e$. See \cite{blmn}, \cite{MN07}, \cite{NT} for further
 investigation. Recently Mendel and Naor \cite{MN112, MN11} studied an another variant of nonlinear
 Dvoretzky's theorem, answering a question by T.~Tao. For example they
 obtained the following: For a metric space $X$ we denote by $\dim_H(X)$
 the Hausdorff dimension of $X$.

  \begin{thm}[{cf.~\cite[Theorem 1.7]{MN11}}]\label{MNDV}There exists a universal constant $c\in
  (0,\infty)$ such that for every $\e \in (0,\infty)$, every compact
   metric space $X$ contains a closed subset $S\subseteq X$ that embeds with distortion
  $2+\e$ in an ultrametric space, and
  \begin{align*}
   \dim_H(S)\geq \frac{c\e}{\log (1/\e)}\dim_H(X).
   \end{align*}
  \end{thm}

  Note that since every separable ultrametric space isometrically embed
  into $\ell_1$, $\ell_2$, and $c_0$ (\cite{Ves2}), the above $S$ embeds into these spaces.

  Applying Theorem \ref{MT} to Theorem \ref{MNDV} we obtain the following
  $\ell_p$-version of nonlinear Dvoretzky's theorem:
  \begin{cor}There exists a universal constant $c\in
  (0,\infty)$ such that for every $\e \in (0,\infty)$, every compact
   metric space $X$ contains a closed subset $S\subseteq X$ that embeds with distortion
  $2+\e$ in $\ell_p$, and
  \begin{align*}
   \dim_H(S)\geq \frac{c\e}{\log (1/\e)}\dim_H(X).
   \end{align*}
   \end{cor}

   Mendel and Naor also obtained the following impossibility result for
   distortion less than $2$:

 \begin{thm}[{cf.~\cite[Theorem 1.8]{MN11}}]\label{hanrei}For every $\alpha>0$ there exists a compact metric space $(X,\dist)$ of Hausdorff
dimension $\alpha$, such that if $S\subseteq X$ embeds into a Hilbert space with distortion strictly smaller
than $2$ then $\dim_H(S) = 0$.
\end{thm}

We shall consider an impossibility problem for the $\ell_p$-version of nonlinear
Dvoretzky's theorem. 

In the proof of Theorem \ref{hanrei} Mendel and Naor used the following
result: Let $G$ be the random graph on $n$-vertices of the
Erd\"os-Reyni model $G(n,1/2)$, i.e., every edge is present
independently with probability $1/2$. From $G$ we construct a metric
space $W_n$ by assigning the distance between each two vertices of $G$ by
$1$ if they are joined by an edge, and $2$ if they are not joined by an edge. Then the obtained metric space $W_n$ satisfies the following
property (\cite{blmn}). There exists $K\in (0,\infty)$ such that for any $n\in \mathbb{N}$
there exists an $n$-point metric space $W_n$ such that for every $\delta
\in (0,1)$ any subset
of $W_n$ of size larger than $2\log_2 n + K (\delta^{-2}\log(2/\delta))^2$
must incur distortion at least $2-\delta$ when embedded into $\ell_2$.

 Bartal, Linial, Mendel, and Naor obtained a similar result for the
    same $W_n$ when considering $\ell_p$ instead of $\ell_2$
    (\cite{blmn2}). Then Charikar and Karagiozova \cite[Theorem 1.3]{ChKa} improved the result in
    \cite{blmn2}: For any $\delta\in (0,1)$ and
    $p\geq 1$, there is a constant $c(p,\delta)$ depending only on $p$
    and $\delta $ such that any subset of $W_n$ of size larger than
       $c(p,\delta) \log n$ must incur distortion at least
       $2-\delta$ when embedded into $\ell_p$. 

    Then applying this result to the proof in \cite[Section
    7.3]{MN11} implies the following: 
   \begin{prop}For every $p\geq 1$ and $\alpha>0$, there exists a compact metric space
    $(X,\dist)$ with $\dim_H(X,\dist)=\alpha$, such that if $S\subseteq
    X$ embeds into $\ell_p$ with distortion strictly smaller than $2$
    then $\dim_H(S)=0$.
    \end{prop}

    The case of the distortion $2$ remains open for any $p\geq 1$.

\begin{ack}\upshape
        The author wish to express his gratitude to Mr. Ryokichi
 Tanaka for discussion. The author also would like to express his thanks
 to Professor Assaf Naor for valuable comments and suggestions which improved the preliminary version of this paper.
 \end{ack}

\end{document}